\documentclass[reqno,12pt]{amsart}
\theoremstyle{plain}
\newtheorem{thm}{Theorem}%[section]
\newtheorem{lemma}[thm]{Lemma} %%Delete [thm] to re-start numbering
\newtheorem{prop}[thm]{Proposition}

\theoremstyle{remark}

\theoremstyle{definition}
\newtheorem{defi}[thm]{Definition}

\newtheorem{obs}[thm]{Observation}

\newtheorem{noname}[thm]{}

\usepackage{amscd,amssymb,comment,epic,eepic,euscript,graphics}
\usepackage{enumerate}
\usepackage[initials]{amsrefs}
\usepackage[all]{xy}
\usepackage{color}

\newcommand\Cpx{{\mathbf C}}
\newcommand\Dc{{\mathcal{D}}}

\newcommand\Exp{{\operatorname{Exp}}}

\newcommand\Mcal{{\mathcal{M}}}
\newcommand\Nats{{\mathbf N}}

\newcommand\supp{\operatorname{supp}}

\begin{document}

\title[]{Upper Triangular Forms and Spectral Orderings in a II$_1$-factor}

\author[Noles]{J. Noles}
\address{Department of Mathematics, Texas A\&M University, College Station, TX, USA.}
\email{jnoles@math.tamu.edu}
%\thanks{\footnotesize ${}^{*}$.}


%\keywords{}

\begin{abstract} 
Dykema, Sukochev and Zanin used a Peano curve covering the support of the Brown measure of an operator $T$ in a diffuse, finite von Neumann algebra to give an ordering to the support of the Brown measure, and create a decomposition  $T = N + Q$, where $N$ is normal and $Q$ is s.o.t.-quasinilpotent.  In this paper we prove that a broader class of measurable functions can be used to order the support of the Brown measure giving normal plus s.o.t.-quasinilpotent decompositions.
\end{abstract}

\maketitle

\section{Introduction and description of results}

We start with a famous theorem of Schur (see for instance \cite{zha}) which will motivate this paper.

\begin{thm}
For every matrix $T \in M_n (\Cpx)$, there exists a unitary matrix $U \in M_n (\Cpx)$ such that $U^{-1}TU$ is an upper triangular matrix.
\end{thm}

The diagonal entries of $U^{-1}TU$ are the eigenvalues of $T$, repeated up to multiplicity, and $U$ can be chosen so that they appear in any order.  Hence each ordering of the spectrum of $T$ gives a decomposition $T = N + Q$, where $N$ is normal and $Q$ is nilpotent.

In \cite{DSZ}, Dykema, Sukochev and Zanin use Haagerup-Schultz projections to prove a related theorem in II$_1$-factors.

\begin{thm}
Let $\Mcal$ be a diffuse, finite von Neumann algebra with normal, faithful, tracial state $\tau$ and let $T \in \Mcal$.  Then there exist $N,Q \in \Mcal$ such that
\begin{enumerate}
\item
$T=N+Q$
\item
the operator $N$ is normal and the Brown measure of $N$ equals that of $T$
\item
The operator $Q$ is s.o.t.-quasinilpotent.
\end{enumerate}
\end{thm}

The proof of Theorem 2 uses a Peano curve $\rho:[0,1] \to \overline{B_{\Vert T \Vert}}$.  The normal operator $N$ is created by taking the trace-preserving conditional expectation onto the von Neumann algebra generated by the Haagerup-Schultz projections of the operator $T$ associated with the sets $\rho([0,t])$ for $t \in [0,1]$.  These projections, along with the normal operator $N$, are determined by the ordering on the support of the Brown measure of $T$ given by $z_1 \leq z_2$ if and only if $\min (\rho^{-1}(z_1)) \leq \min (\rho^{-1} (z_2))$.  Theorem 2 generalizes the idea of using an oredering of the spectrum of the operator $T$ to write it as an uppertriangular form.

In this paper we will further generalize the idea of spectral orderings from the finite dimensional case to II$_1$-factors.  We show that normal plus s.o.t.-quasinilpotent decompositions are generated not only by continuous orderings, but by a large class of measurable orderings.

\begin{thm}
Let $\Mcal$ be a $II_1$-factor and $T \in \Mcal$.  Let $\nu_T$ be the Brown measure of $T$ and for a Borel set $B \subset \overline{B_{\Vert T \Vert}}$, let $P_T(B)$ be the Haagerup-Schultz projection for the operator $T$ associated to the set $B$.  Let $\psi : [0,1] \to \overline{B_{\Vert T \Vert}}$ be a Borel measurable function such that $\psi ([0,t])$ is Borel for all $t \in [0,1]$, $\{ z \in  \overline{B_{\Vert T \Vert}}: \psi ^{-1} (z)\, has\,\,a\,\,minimum \}$ is Borel, and 
$$\nu_T (\{ z \in \overline{B_{\Vert T \Vert}}: \psi ^{-1} (z)\, has\,\,a\,\,minimum \}) = 1.$$

Then there exists a spectral measure $E$ supported on $\supp (\nu_T)$ such that
\begin{enumerate}
\item $E(\psi ([0,t])) = P_T (\psi ([0,t]))$ for all $t \in [0,1]$,
\item $\tau (E (B)) = \nu_T (B)$ for all Borel $B \subset \overline{B_{\Vert T \Vert}}$, and
\item $T - \int_{\Cpx} zdE$ is s.o.t.-quasinilpotent.
\end{enumerate}
\end{thm}

In particular the conclusion holds if $\psi$ is continuous or is a Borel isomorphism.  We leave open the following question: Given a function $\varphi$ which satisfies the hypotheses of Theorem 3, does there exist a Borel ismorphism $\psi$ such that $\varphi$ and $\psi$ generate the same spectral measure?

Note that part 2 of theorem 3  implies that $\int_{\Cpx} zdE$ and $T$ have that same Brown measure.

\section{Background: Conditional expectation, Brown measure, Haagerup-Schultz projections and s.o.t.-quasinilpotent operators}

This section includes some background necessary for the proof of Theorem 3.  Throughout this section $\Mcal$ is a II$_1$-factor with trace $\tau$, and $T \in \Mcal$.

\begin{defi}Let $\mathcal{N}$ be a von Neumann subalgebra of $\Mcal$.  Then there exists a unique trace-preserving faithful normal linear map $\mathbb{E}_{\mathcal{N}} : \Mcal \to \mathcal{N}$.  $\mathbb{E}_{\mathcal{N}}$ satisfies the properties
\begin{enumerate}
\item $\mathbb{E}_{\mathcal{N}}$ is completely positive and unital
\item For any $T_1,T_2 \in \mathcal{N}$ and any $S \in \Mcal$, $\mathbb{E}_{\mathcal{N}}(T_1ST_2) = T_1\mathbb{E}_{\mathcal{N}}(S)T_2.$
\end{enumerate}
The map $\mathbb{E}_{\mathcal{N}}$ is called the {\bf conditional expectation} of $\Mcal$ onto $\mathcal{N}$.
\end{defi}

\begin{defi}
In \cite{Br86}, Brown constructed and proved unique a probability measure $\nu_T$ supported on a compact subset of spec$(T)$ such that for any $\lambda \in \Cpx$,
$$\tau(\log(\vert T-\lambda \vert)) = \int_{\Cpx} \log (\vert z - \lambda\vert) d\nu_T(z).$$
$\nu_T$ is called the {\bf Brown measure} of T.
\end{defi}

In the case that $T$ is normal, Brown's construction gives $\nu_T = \tau \circ E$, where $E$ is the projection valued spectral decomposition measure of $T$.

The following theorem of Haagerup and Schultz is the cornerstone of our proof.

\begin{thm}
Let $\Mcal$ be a II$_1$-factor with trace $\tau$ and let $T \in \Mcal$.  For every Borel set $B \subset \Cpx$, there exists a unique projection $P_T(B) \in \Mcal$ such that\
\begin{enumerate}
\item
$\tau (P_T(B)) = \nu_T (B)$, where $\nu_T$ is the Brown measure of $T$,
\item
$TP_T(B) = P_T(B)TP_T(B)$,
\item
if $P_T(B) \neq 0$, then the Brown measure of $TP_T(B)$ considered as an element of $P_T(B) \Mcal P_T(B)$ is concentrated in $B$ and
\item
if $P_T(B) \neq 1$, then the Brown measure of $(1-P_T(B))T$, considered as an element of $(1-P_T(B))\Mcal (1-P_T(B))$, is concentrated in $\Cpx \setminus B$.
\end{enumerate}
Moreover, $P_T(B)$ is $T$-hyperinvariant and if $B_1 \subset B_2 \subset \Cpx$ are Borel sets, then $P_T(B_1) \leq P_T(B_2)$.
\end{thm}

The projection $P_T(B)$ in Theorem 6 is called the Haagerup-Schultz projection of $T$ associated to the set $B$.

The following two results, from \cite{DSZ2} and \cite{HS2}, respectively, will be crucial to the proof of part 3 of theorem 3.

\begin{lemma}
For any increasing, right-continuous family of $T$-invariant projections $(q_t)_{0\le t\le 1}$ with $q_0=0$ and $q_1=1$,
letting $\Dc$ be the von Neumann algebra generated by the set of all the $q_t$ and $\Dc '$ be the relative commutant of $\Dc$ in $\Mcal$, and letting $\Exp_{\Dc '}$ be the $\tau$ preserving conditional expectation, 
the Fuglede--Kadison determinants of $T$ and $\Exp_{\Dc'}(T)$ agree.
Since the same is true for $T-\lambda$ and $\Exp_{\Dc'}(T)-\lambda$ 
for all complex numbers $\lambda$,
we have that the Brown measures of $T$ and $\Exp_{\Dc'}(T)$ agree.
\end{lemma}

\begin{thm}
If $T \in \Mcal$, and if $p \in \Mcal$ is a projection such that $Tp = pTp$, so that we may write $T = \left(\begin{matrix}A&B\\0&C\end{matrix}\right), where $A = Tp and $C = (1-p)T$, then
$$\Delta_{\Mcal} (T) = \Delta_{p \Mcal p} (A)^{\tau(p)} \Delta_{(1-p) \Mcal (1-p)} (C)^{\tau (1-p)}$$
and
$$\nu_T = \tau(p) \nu_A + \tau (1-p) \nu_C,$$
where $A$ is considered as an element of $p\Mcal p$ and $C$ is considered as an element of $(1-p) \Mcal (1-p)$.
\end{thm}

\begin{defi}
It was shown in \cite{HS09} that for any $T \in \Mcal$, $((T^*)^nT^n)^{1/2n}$ converges in the strong operator topology as $n$ approaches $\infty$.  An operator $T$ is called {\bf s.o.t.-quasinilpotent} if $((T^*)^nT^n)^{1/2n} \to 0$ in the strong operator topology as $n \to \infty$.
\end{defi}

It was also shown in \cite{HS09} that $T$ is s.o.t.-quasinilpotent if and only if the Brown measure of $T$ is concentrated at $0$.

We will also need a characterization from \cite{HS09} of the Haagerup-Schultz projection of $T$ associated with the ball $\overline{B_r} = \{\vert z \vert \leq r \}$.

\begin{noname}
Suppose $\Mcal \leq \mathcal{B}(\mathcal{H})$.  Define a subspace $\mathcal{H}_r$ of $\mathcal{H}$ by
$$\mathcal{H}_r = \{ \xi \in \mathcal{H}: \exists \xi_n \to \xi,\,\, \mathrm{with\,} \limsup_{n \to \infty} \Vert T^n \xi_n \Vert^{1/n} \leq r \}.$$
Then the projection onto $\mathcal{H}_r$ is equal to $P_T(\overline{B_r})$.
\end{noname}

\section{Construction of the spectral measure $E$}

Throughout this section, $\Mcal$, $T$, $\nu_T$, $P_T$ and $\psi$ will be as described in Theorem 3, $Z$ will denote $\{ z \in  \overline{B_{\Vert T \Vert}} : \psi^{-1} (z)\, \mathrm{ has\,\,a\,\,minimum} \}$ and $Y$ will denote $\overline{B_{\Vert T \Vert}} \setminus Z$.

We first define a Borel measure on the unit interval which will be useful in later proofs.

\begin{lemma}
Let $X = \{ \min (\psi^{-1}(z)): z \in \overline{B_{\Vert T \Vert}}\}$.  If $b \subset [0,1]$ is Borel, then $\psi (b \cap X)$ is Borel.
\end{lemma}
\begin{proof}
Note first that, for $t \in (0,1]$, we have $\psi ([0,t] \cap X) = \psi ([0,t]) \setminus Y$ and $\psi ([0,t) \cap X) = \psi ([0,t)) \setminus Y$, and these sets are Borel.  Now, since $\psi$ restricted to $X$ is an injection, we have $\psi ((\alpha , \beta) \cap X) = \psi ([0, \beta) \cap X) \setminus \psi ([0, \alpha] \cap X)$ which is Borel.  Since $[0,1]$ is second countable, an arbitrary open set $v = \bigcup_{n \in \Nats}u_n$ is the countable union of open intervals so that $\psi (v \cap X) = \psi (\bigcup_{n \in \Nats}(u_n \cap X)) = \bigcup_{n \in \Nats}(\psi (u_n \cap X))$ is Borel.

To complete the proof, we show that the collection of sets 
$$S = \{ b \subset [0,1] : \psi (b \cap X) \mathrm{\, is \, Borel} \}$$
 forms a $\sigma$-algebra.  Suppose that $\psi (b \cap X)$ is Borel.  Then $\psi (b^c \cap X) = \psi (X \setminus (b\cap X)) = Z \setminus \psi (b\cap X)$ is Borel.  Now suppose that $(b_n)_{n \in \Nats}\subset S$.  Then $\bigcup_{n\in \Nats} b_n \in S$ by the same argument used for open sets, and we are done.
\end{proof}

We now define $\mu (b) = \nu_T (\psi (b \cap X))$ for any Borel set $b \subset [0,1]$.  It is clear that $\mu$ is countably additive, and hence a Borel probability measure on $[0,1]$.  That $\mu$ is a regular measure follows from Theorem 1.1 of \cite{bil}.

\begin{obs}
For any Borel set $B \subset \overline{B_{\Vert T \Vert}}$, $\mu (\psi^{-1}(B)) = \nu_T (B)$.
\end{obs}
\begin{proof}
Since $\psi$ is a bijection from $X$ to $Z$ we have
$$\mu (\psi^{-1} (B)) = \nu_T (\psi (\psi^{-1} (B) \cap X)) 
= \nu_T (B \cap Z) = \nu_T (B)$$
\end{proof}

Prior to constructing the spectral measure, we will need a map from the open subsets of the closed unit interval to the set of projections in $\Mcal$.  For an open interval, define 
\begin{align*}
&F(\emptyset) = 0\\
&F((\alpha , \beta )) = P_T(\psi ([0,\beta))) - P_T(\psi ([0,\alpha]))\\
&F([0,\beta)) = P_T (\psi ([0,\beta)))\\
&F((\alpha,1]) = 1-P_T (\psi ([0,\alpha])).
\end{align*}

Since $P_T (\psi ([0,t]))$ and $P_T (\psi ([0,t)))$ are increasing in $t$, it follows that $F(u)$ is increasing in $u$, and $F(u_1)F(u_2) = 0$ if $u_1 \cap u_2 = \emptyset$.  For $u_1 = (\alpha_1,\beta_1)$ and $u_2 = (\alpha_2,\beta_2)$ with $\alpha_1 \leq \alpha_2 \leq \beta_1 \leq \beta_2$, 
\begin{align*}
F(u_1)F(u_2) &= (P_T(\psi ([0,\beta_1))) - P_T(\psi ([0,\alpha_1])))(P_T( \psi ([0,\beta_2))) - P_T( \psi ([0,\alpha_2])))\\
&= P_T( \psi ([0,\beta_1))) - P_T(\psi ([0,\alpha_2])) - P_T(\psi ([0,\alpha_1])) + P_T(\psi ([0,\alpha_1]))\\
&=F(u_1 \cap u_2).
\end{align*}
Hence for any open intervals $u_1$ and $u_2$, $F(u_1)F(u_2) = F(u_1 \cap u_2)$.

For an arbitrary open set $v \subset [0,1]$, we first write $v = \bigcup_{n \in \Nats} u_n$, where the $u_n$ are pairwise disjoint, and all nonempty $u_n$ are open intervals.  Then $\sum_{n \in \Nats} F(u_n)$ converges to a projection in the strong operator topology.  We define $F(v) = \sum_{n \in \Nats} F(u_n)$.  Multiplication of the series and application of the corresponding result for intervals gives us $F(v_1)F(v_2) = F(v_1 \cap v_2)$ for open sets $v_1,v_2 \subset [0,1]$.

\begin{obs}
For any open set $v \subset [0,1]$, $\tau (F(v)) = \mu (v)$.
\end{obs}

\begin{proof}
For an open interval $u = (\alpha, \beta)$, we have 
\begin{align*}
\tau (F(u)) &= \tau (P_T (\psi ([0,\beta))) - P_T (\psi ([0,\alpha])) \\
&= \nu_T (\psi ([0,\beta))) - \nu_T (\psi ([0, \alpha])) \\
&= \mu ([0,\beta)) - \mu ([0,\alpha]) \\
&= \mu (u).
\end{align*}
The observation follows from additivity of $\mu$, $F$ and $\tau$.
\end{proof}

We are now ready to define the spectral measure $E$.  
For any Borel set $B \subset \overline{B_{\Vert T \Vert}}$, define 
\[
E(B) = \wedge \{ F(v) : v \, is \, open\, and \, \psi^{-1} (B) \subset v \} .
\]
Note that $E$ is increasing and that the range of $E$ is contained in the von Neumann algebra generated by the projections $P_T (\psi ([0,t]))$ for $t \in [0,1]$, which is commutative.  We will prove later that $E$ defines a spectral measure.

\begin{prop}
For any Borel set $B \subset \overline{B_{\Vert T \Vert}}$, $\tau (E(B)) = \nu_T (B)$.
\end{prop}

\begin{proof}
Let $\epsilon > 0$ be given.  There exist open sets $v_1,v_2 \subset [0,1]$ such that 
\begin{enumerate}
\item
$\psi^{-1} (B) \subset v_1$ and $\mu (v_1) - \mu (\psi^{-1} (B)) < \epsilon$, and
\item
$\psi^{-1} (B) \subset v_2$ and $\tau (F(v_2)) - \tau (E (B)) < \epsilon.$
\end{enumerate}
Applying Observations 12 and 13 to (1), we have 
\[
\tau (E(B)) - \nu_T (B) \leq \tau(F(v_1)) - \nu_T (B) = \mu (v_1) - \mu (\psi^{-1} (B)) < \epsilon.
\]
Applying Observations 12 and 13 to (2) gives 
\[
\nu_T (B) - \tau (E (B)) = \mu (\psi^{-1} (B)) - \tau (E (B)) 
\]
\[
\leq \mu (v_2) - \tau (E (B)) = \tau (F(v_2)) - \tau (E (B)) <\epsilon.
\]
Hence we have $\vert \tau (E(B)) - \nu_T (B) \vert < \epsilon$, and we are done.
\end{proof}

%\begin{obs}
%For any Borel set $B \subset \overline{B_{\Vert T \Vert}}$, there exists a decreasing sequence of open subsets $(V_n)$ of $\overline{B_{\Vert T \Vert}}$ such that $B \subset \cap_{n \in \Nats} V_n$ and $E(V_n) \to E(B)$ in the strong operator topology.
%\end{obs}

%\begin{proof}
%By claim 5, there exists a decreasing sequence $(V_n)$ of open subsets of $\overline{B_{\Vert T \Vert}}$ such that $B \subset \cap_{n \in \Nats} V_n$ and $(\tau \circ E)(V_n) - (\tau \circ E)(B) < \frac{1}{n}$.  $E(V_n)$ converges in the strong operator topology to a superprojection $P$ of $E(B)$ and $(\tau \circ E)(V_n)$ converges to $(\tau \circ E)(B)$, so we conclude $P = E(B)$.
%\end{proof}

\begin{lemma}
If $B_1$ and $B_2$ are Borel subsets of $\overline{B_{\Vert T \Vert}}$, then $E(B_1)E(B_2) = E(B_1 \cap B_2)$.
\end{lemma}

\begin{proof}
Noting that whenever $v_1$ is an open set containing $\psi^{-1}(B_1)$ and $v_2$ is an open set containing $\psi^{-1}(B_2)$, $v_1 \cap v_2$ is an open set containing $\psi^{-1}(B_1) \cap \psi^{-1}(B_2)$, we have 
\begin{align*}
&E(B_1 \cap B_2) = \wedge \{ F(v) : v \, open, \psi^{-1}(B_1 \cap B_2) \subset v \} \\
&= \wedge \{ F(v) : v \, open, \psi^{-1}(B_1) \cap \psi^{-1}(B_2) \subset v \} \\
&\leq \wedge \{ F(v_1 \cap v_2) : v_1,v_2 \, open,\psi^{-1}(B_1) \subset v_1,\psi^{-1}(B_2) \subset v_2 \} \\
&= \wedge \{ F(v_1)F(v_2) : v_1,v_2 \, open,\psi^{-1}(B_1) \subset v_1,\psi^{-1}(B_2) \subset v_2 \} \\
&= \wedge \{ F(v_1) : v_1 \, open,\psi^{-1}(B_1) \subset v_1 \} \wedge \{ F(v_2) : v_2 \, open,\psi^{-1}(B_2) \subset v_2 \} \\
&= E(B_1)E(B_2).
\end{align*}

Now let $\epsilon >0$ be given.  There exist open subsets $v,\tilde{v_1},\tilde{v_2}$ of $[0,1]$ such that
\begin{enumerate}
\item
$\psi^{-1}(B_1 \cap B_2) \subset v$ and $\mu (v \setminus \psi^{-1}(B_1 \cap B_2)) < \epsilon$,
\item
$a_1 = \psi^{-1} (B_1) \setminus \psi^{-1} (B_1 \cap B_2) \subset \tilde{v_1}$ and $\mu (\tilde{v_1} \setminus a_1) < \epsilon$, and
\item
$a_2 = \psi^{-1} (B_2) \setminus \psi^{-1} (B_1 \cap B_2) \subset \tilde{v_2}$ and $\mu (\tilde{v_2} \setminus a_2) < \epsilon$.
\end{enumerate}
Let $v_i = \tilde{v_i} \cup v$ for $i=1,2$.  Then $v_1$ is an open set containing $\psi^{-1}(B_1)$ and $v_2$ is an open set containing $\psi^{-1}(B_2)$.   We have 
\[
\mu(v_1 \cap v_2 \setminus \psi^{-1}(B_1 \cap B_2)) \leq \mu(v \setminus \psi^{-1}(B_1 \cap B_2)) + \mu(\tilde{v_1} \cap \tilde{v_2} \setminus \psi^{-1}(B_1 \cap B_2)).
\]
Observing that $a_1 \cap a_2 = \emptyset$ and 
\[\tilde{v_1} \cap \tilde{v_2} = (a_1 \cap a_2) \cup ((\tilde{v_1} \setminus a_1) \cap a_2) \cup ((\tilde{v_2} \setminus a_2) \cap a_1) \cup ((\tilde{v_1} \setminus a_1) \cap (\tilde{v_2} \setminus a_2))
\]
we have 
\[
\mu((v_1 \cap v_2) \setminus \psi^{-1}(B_1 \cap B_2)) < 4 \epsilon.
\]
Applying Observations 12 and 13 and Proposition 14, we have 
\begin{align*}
\tau (E(B_1)E(B_2)) - \tau(E(B_1 \cap B_2)) &\leq \tau(F(v_1)F(v_2)) - \tau(E(B_1 \cap B_2)) \\
&= \tau(F(v_1 \cap v_2)) - \tau(E(B_1 \cap B_2)) \\
&< 4 \epsilon,
\end{align*}
and we conclude $E(B_1)E(B_2) = E(B_1 \cap B_2)$.
\end{proof}

\begin{lemma}
$E$ is countably additive on disjoint sets, where convergence of the series is in the strong operator topology.
\end{lemma}

\begin{proof}
Suppose $(B_n)_{n \in \Nats}$ is a countable collection of disjoint Borel subsets of $\overline{B_{\Vert T \Vert}}$.  By claim 7, $E(B_i)E(B_j) = 0$ if $i \neq j$.  Then $E(\bigcup_{n \in \Nats}B_n)$ is a superprojection of each $E(B_n)$, and hence a superprojection of $\sum_{n \in \Nats} E(B_n)$.  Also, $\tau(E(\bigcup_{n \in \Nats}B_n)) = \nu_T (\bigcup_{n \in \Nats} B_n) = \tau(\sum_{n \in \Nats} E(B_n))$.  We conclude $E(\bigcup_{n \in \Nats}B_n) = \sum_{n \in \Nats} E(B_n)$.
\end{proof}

We are now ready to show that $E$ is a spectral measure supported on $\supp (\nu_T)$.

\begin{proof}
We must show three things:
\begin{enumerate}
\item
$E(\emptyset) = 0$ and $E(\supp (\nu_T)) = 1$
\item
$E(B_1 \cap B_2) = E(B_1)E(B_2)$ for Borel sets $B_1,B_2$, and
\item
if $\Mcal$ acts on a Hilbert space $\mathcal{H}$, and $x,y \in \mathcal{H}$, then $\eta (B) = \langle E(B) x,y \rangle$ defines a regular Borel measure on $\Cpx$.
\end{enumerate}

\begin{enumerate}
\item
Follows from Proposition 14, since $\tau (E(\emptyset)) = 0$ and $\tau (E(\supp (\nu_T))) = 1$.
\item
Was proven as Lemma 15.
\item
That $\eta$ is countably additive on disjoint sets follows from Lemma 15.  Regularity of $\eta$ follows from Theorem 1.1 of \cite{bil}.
\end{enumerate}
\end{proof}

\section{Proof of theorem 3}

We first establish several results which will be used to prove Part 3.  Throughout this section, $\Mcal$, $T$, and $\psi$ are as described in Theorem 3, and $\mu$, $E$ and $E_v$ are as defined in Section 3.  $\Mcal$ acts on a Hilbert space $H$.

We now show that $\int_{\Cpx} zdE$ is the norm limit of conditional expectations onto an increasing sequence of abelian von Neumann algebras.  For each $n$, divide the $3 \Vert T \Vert$ by $3 \Vert T \Vert$ square centered at $0$ into $2^n$ by $2^n$ squares of equal size indexed $(A_{n,k})_{k=1}^{2^{2n}}$, $k$ increasing to the right then down.  Include in each $A_{n,k}$ the top and left edge, excluding the bottom-left and top-right corners, so that for each $n$, $A_{n,k} \cap A_{n,j} = \emptyset$ whenever $j \neq k$ and $\overline{B_{\Vert T \Vert}} \subset \cup_{k=1}^{2^{2n}} A_{n,k}$.  Let $D_n$ be the von Neumann algebra generated by the (orthogonal) projections $(E(A_{n,k}))_{k=1}^{2^{2n}}$.

\begin{prop}
Let $\mathbb{E}_{D_n}(T)$ denote the conditional expectation of $T$ onto $D_n$.  Then $\mathbb{E}_{D_n}(T)$ converges in norm as $n \to \infty$ to $\int_{\Cpx} zdE$.
\end{prop}

\begin{proof}
Observe that
$$\mathbb{E}_{D_n}(T) = \sum\limits_{\substack{1 \leq k \leq 2^{2n}\\ \tau(E(A_{n,k})) \neq 0}} \frac{\tau(E(A_{n,k})TE(A_{n,k}))}{\tau (E(A_{n,k}))}E(A_{n,k}).$$
Applying Brown's analog of Lidskii's theorem (see \cite{Br86}) gives
$$\mathbb{E}_{D_n}(T) = \sum\limits_{\substack{1 \leq k \leq 2^{2n} \\ \nu_T (A_{n,k}) \neq 0}} \frac{\int_{A_{n,k}} zd \nu_T (z)}{\nu_T (A_{n,k})}E(A_{n,k}).$$

For each $n$, define $$f_n (w) = \sum\limits_{\substack{1 \leq k \leq 2^{2n} \\ \nu_T (A_{n,k}) \neq 0}} \frac{\int_{A_{n,k}} zd \nu_T (z)}{\nu_T (A_{n,k})}\chi_{A_{n,k}}(w) + \sum\limits_{\substack{1 \leq k \leq 2^{2n} \\ \nu_T (A_{n,k}) = 0}} \frac{\int_{A_{n,k}} zd m (z)}{m (A_{n,k})}\chi_{A_{n,k}}(w),$$
where $m$ is the Lebesgue measure on $\Cpx$.

Since $\nu_T(A_{n,k}) = 0$ implies $E(A_{n,k}) = 0$, $\int_{\Cpx} f_n dE = \mathbb{E}_{D_n}(T)$.  Note that $f_n$ converges uniformly on $\supp(E)$ to the inclusion function $f(z) = z$.  Hence $\int_{\Cpx} f_n dE$ converges in norm to $\int_{\Cpx} zdE$, and we are done.
\end{proof}

Let $D$ be the von Neumann algebra generated by $(E(\psi ([0,t])))_{t \in [0,1]}$ (or equivalently by $\bigcup_{n=1}^{\infty} D_n$).

\begin{prop}
Suppose that $T \in D'$ and $B \subset \overline{B_{\Vert T \Vert}}$ is Borel with $\nu_T (B) \neq 0$.  Then the Brown measure of $E(B)TE(B)$, considered as an element of $E(B) \Mcal E(B)$, is concentrated in $B$.
\end{prop}

\begin{proof}
We begin by observing that for any open $v \subset [0,1]$, with $\tau (F(v)) \neq 0$, $F(v) \in D$ and if $v = (\alpha,\beta)$ is an open interval, then $\nu_{TF(v)}$ is concentrated in $\psi ([0,\beta)) \setminus \psi ([0,\alpha])$, and hence is also concentrated in $\psi ((\alpha,\beta)) \cap Z$, where $Z$ is as described in Section 3.  Thus $\nu_{TF(v)}$ is concentrated in $\psi ((\alpha,\beta)\cap X)$.

Now suppose that $v = \bigcup_{n=1}^{\infty} u_n$ where all nonempty $u_n$ are pairwise disjoint open intervals.  Let $\epsilon > 0$ be given.  Let $N$ be so large that 
\[
\tau \left( \sum_{n = 1}^N F(u_n) \right) > \tau(F(v)) (1 - \epsilon).
\]
Then, since each $F(u_n)$ commutes with $T$, Theorem 8 gives 
\[
\nu_{TF(v)} = \frac{1}{\tau(F(v))} \left( \sum_{n=1}^N \tau(F(u_n) )\nu_{TF(u_n)} + \tau \left( \sum_{n=N+1}^\infty F(u_n) \right) \nu_{(\sum_{n=N+1}^\infty F(u_n))T} 
\right).
\]
Hence, since each $\nu_{TF(u_n)}$ is concentrated in $\psi (u_n \cap X) \subset \psi (v \cap X)$, we have
\[
\nu_{TF(v)}(\psi (v \cap X)) \geq \frac{1}{\tau(F(v))} \left( \sum_{n=1}^N \tau(F(u_n)\right) \nu_{TF(u_n)} (\psi (v \cap X)) > 1-\epsilon,
\]
so that $\nu_{TF(v)}$ is concentrated in $\psi (v \cap X)$.

Now observe that when $v$ is an open set containing $\psi^{-1}(B)$, since 
\[
\nu_{TF(v)} = \frac{1}{\tau(F(v))}(\tau (E(B)) \nu_{TE(B)} + \tau(F(v) - E(B)) \nu_{(F(v) - E(B))T}), 
\]
$\nu_{TE(B)}$ is concentrated in $\psi (v \cap X)$.

Choose an open set $v \subset [0,1]$ such that $\psi^{-1}(B) \subset v$ and $\mu(v) - \mu(\psi^{-1}(B)) < \epsilon$.  Then using Theorem 7 and Observation 11, 
\begin{align*}
\epsilon &> \nu_T (\psi (v \cap X)) - \nu_T (B)\\
 &= \tau (E(B)) \nu_{TE(B)} (\psi (v \cap X) \setminus B) + (1 - \tau (E(B))) \nu_{(1-E(B))T} (\psi (v \cap X) \setminus B)\\
 &\geq \tau (E(B)) \nu_{TE(B)} (\psi (v \cap X) \setminus B).\\
\end{align*}
Hence
$$\tau (E(B)) - \epsilon < \tau (E(B))(1 - \nu_{TE(B)}(\psi (v \cap X) \setminus B)) = \tau (E(B))(\nu_{TE(B)}(B)).$$
Thus
$$1 - \frac{\epsilon}{\tau (E(B))} < \nu_{TE(B)} (B).$$
Letting $\epsilon$ tend to $0$ gives the desired result.
\end{proof}

\begin{lemma}
If $T \in D'$, then the Brown measure of $T - \mathbb{E}_{D_n}(T)$ is supported in the ball of radius $\frac{6 \sqrt{2} \Vert T \Vert}{2^n}$.
\end{lemma}

\begin{proof}
The key observation is that for any $\alpha \in \Cpx$, if $\nu_{T-\alpha}$ is the Brown measure of $T - \alpha$, then for any Borel set $B \subset \Cpx$, $\nu_{T-\alpha}(B) = \nu_T (B-\alpha)$.  Since whenever $E(A_{n,k}) \neq 0$ the Brown measure of $TE(A_{n,k})$ is supported in $A_{n,k}$, the Brown measure of $(T-\frac{\tau (TE(A_{n,k}))}{\tau(E(A_{n,k}))})E(A_{n,k})$ is supported in the square centered at 0 with edge length $\frac{6 \Vert T \Vert}{2^n}$.  We complete the proof by observing that $T- \mathbb{E}_{D_n} (T) = \sum_{k=1}^{2^{2n}} \left (T- \frac{\tau (TE(A_{n,k}))}{\tau(E(A_{n,k}))} \right )E(A_{n,k})$ and applying Theorem 8 to compute the Brown measure of the sum.
\end{proof}

We now are ready to prove Theorem 3.

\begin{proof}
\begin{enumerate}
\item
Whenever $v$ is an open set containing $\psi^{-1} (\psi ([0,t]))$, there exists $\epsilon > 0$ such that $[0,t+\epsilon) \subset v$ so we see that 
\[
P_T (\psi ([0,t])) \leq F([0,t + \epsilon)) \leq F(v).
\]
Hence we see that 
\[
P_T (\psi ([0,t])) \leq E(\psi ([0,t])).
\]
By Proposition 14 and Theorem 6, 
\[
\tau (P_T (\psi ([0,t]))) = \tau (E(\psi ([0,t])))
\]
so that 
\[
P_T (\psi ([0,t])) = E(\psi ([0,t])).
\]
\item
Was proven as Proposition 14.
\item
We show this first in the case that $T \in D'$.  Observe from the proof of Proposition 17 that $\Vert \mathbb{E}_D (T) - \mathbb{E}_{D_n} (T) \Vert \leq \frac{3 \sqrt{2} \Vert T \Vert}{2^n}$.  The rest of this argument is taken from the proof of Lemma 24 in \cite{DSZ}.

We assume without loss of generality that $\Vert T \Vert \leq 1/2$.  Fix $n \in \Nats$ and a unit vector $\xi \in H$.  By assumption $T \in D'$, so we have 
$$(T - \mathbb{E}_D (T))^{2m} = \sum_{k=0}^{2m} (-1)^k {{2m} \choose k} (\mathbb{E}_D (T) - \mathbb{E}_{D_n} (T))^{2m-k} (T - \mathbb{E}_{D_n} (T))^k.$$

Since $\Vert T \Vert \leq 1/2$, both $\mathbb{E}_D (T) - \mathbb{E}_{D_n} (T)$ and $T - \mathbb{E}_{D_n} (T)$ are contractions.  For $k \leq m$ and any $\eta \in H$, we have 
$$\Vert (\mathbb{E}_D (T) - \mathbb{E}_{D_n} (T))^{2m-k}(T - \mathbb{E}_{D_n} (T))^k \eta \Vert_H \leq \Vert \mathbb{E}_D (T) - \mathbb{E}_{D_n} (T) \Vert^m.$$
For $k > m$ and any $\eta \in H$ we have 
$$\Vert (\mathbb{E}_D (T) - \mathbb{E}_{D_n} (T))^{2m-k}(T- \mathbb{E}_{D_n} (T))^k \eta \Vert_H \leq \Vert (T - \mathbb{E}_{D_n} (T))^m \eta \Vert_H.$$
Hence for any $\eta \in H$,
\begin{equation}
\Vert (T- \mathbb{E}_D (T))^{2m} \eta \Vert_H \leq 2^{2m} \max \left\{ \left (\frac{3 \sqrt{2} \Vert T \Vert}{2^n} \right )^m,\Vert (T - \mathbb{E}_{D_n} (T))^m \eta \Vert_H \right\}.
\end{equation}

By Lemma 19, the Brown measure of $T - \mathbb{E}_{D_n} (T)$ is supported in the ball of radius $\frac{6 \sqrt{2} \Vert T \Vert }{2^n}$ centered at 0.  By the Haagerup-Schultz characterization (10), there exists a sequence $\xi_m \to \xi$ such that $\Vert \xi_m \Vert_H = 1$ and 
$$\limsup_{m \to \infty} \Vert (T - \mathbb{E}_{D_n} (T))^m \xi_m \Vert_H^{1/m} \leq \frac{6 \sqrt{2} \Vert T \Vert}{2^n}.$$
Hence there exists $M$ (depending on $n$) such that 
$$\Vert (T -\mathbb{E}_{D_n} (T))^m \xi_m \Vert_H \leq \left (\frac{7 \sqrt{2} \Vert T \Vert}{2^n} \right )^m,\;\; m>M.$$
Taking $\eta = \xi_m$ in (1), we have
$$\Vert (T - \mathbb{E}_D (T))^{2m} \xi_m \Vert_H^{1/m} \leq \frac{28 \sqrt{2} \Vert T \Vert}{2^n},\;\; m>M.$$
Since $\xi$ was arbitrary, it follows from characterization (10) that the Brown measure of $(T - \mathbb{E}_D (T))^2$ is supported in the ball of radius $\frac{28 \sqrt{2} \Vert T \Vert}{2^n}$ centered at 0.  Letting $n \to \infty$, we obtain that the Brown measure of $T - \mathbb{E}_D (T)$ is $\delta_0$.

For $T \notin D'$, we first show that $P_T(\psi ([0,t])) = P_{\mathbb{E}_{D'}(T)}(\psi([0,t]))$ for all $t \in [0,1]$.  For any $t$, $P_T(\psi([0,t])) \in D$, so $$\mathbb{E}_{D'}(T) P_T(\psi([0,t])) =  P_T(\psi([0,t])) \mathbb{E}_{D'}(T) P_T(\psi([0,t])).$$ 
By Lemma 7, $T$ and $\mathbb{E}_{D'}(T)$ have the same Brown measure, so we have for all $t$
$$\tau(P_T(\psi([0,t]))) = \nu_T(\psi([0,t])) = \nu_{\mathbb{E}_{D'}(T)}(\psi([0,t])).$$
For any $s,t \in [0,1]$ $P_T(\psi([0,s]))$ is $TP_T(\psi([0,t]))$ invariant, so by Lemma 7 $TP_T(\psi([0,t]))$ and $\mathbb{E}_{D'}(TP_T(\psi([0,t])))$ have the same Brown measure for any $t$, so whenever $P_T(\psi([0,t])) \neq 0$ we have
$$\nu_{\mathbb{E}_{D'}(T)P_T(\psi([0,t]))} = \nu_{\mathbb{E}_{D'}(TP_T(\psi([0,t])))} = \nu_{TP_T(\psi([0,t]))}$$
is supported in $\psi([0,t])$.  Similarly $P_T(\psi([0,s]))$ is $(1-P_T(\psi([0,t])))T$ invariant for all $s,t \in [0,1]$, so $(1-P_T(\psi([0,t])))T$ and $\mathbb{E}_{D'}((1-P_T(\psi([0,t])))T) = (1-P_T(\psi([0,t])))\mathbb{E}_{D'}(T)$ have the same Brown measure, which is supported in $\Cpx \setminus \psi([0,t])$ whenever $P_T(\psi([0,t])) \neq 1$.  Hence by Theorem 6 $P_T(\psi([0,t]))$ is the Haagerup-Schultz projection of $\mathbb{E}_{D'}(T)$ associated with the set $\psi([0,t])$.

Since $P_T(\psi ([0,t])) = P_{\mathbb{E}_{D'}(T)}(\psi([0,t]))$ for all $t \in [0,1]$, we see that $\psi$ generates the same spectral measure $E$ and abelian subalgebra $D$ for both $T$ and $\mathbb{E}_{D'}(T)$.  Applying Lemma 7 we have ${T-\int_\Cpx z dE}$ and ${\mathbb{E}_{D'}(T)-\int_\Cpx zdE}$ have the same Brown measure, which we have shown is $\delta_0$.
\end{enumerate}
\end{proof}

\end{document}